\documentclass[letterpaper, 12 pt, onecolumn, journal]{IEEEtran}%

\usepackage{verbatim}
\usepackage{color}
\usepackage{amsmath,amsthm,amssymb}
\usepackage{url}
\usepackage{cite}
\usepackage{enumitem}

\newtheorem{theorem}{Theorem}[section]
\newtheorem{definition}[theorem]{Definition}
\newtheorem{lemma}[theorem]{Lemma}

\theoremstyle{definition}

\theoremstyle{remark}
 
\theoremstyle{remark}

\newcommand{\jmap}[3]{#1: #2 \rightarrow #3}
\newcommand{\real}{\mathbb{R}}

\DeclareMathOperator*{\maxi}{\text{Maximize}}

\newcommand{\expb}[1]{\exp\left({#1}\right)}
\newcommand{\logb}[1]{\log\left({#1}\right)}
\newcommand{\norm}[1]{\left\lVert {#1} \right\rVert}

\newcommand{\One}{\mathbf{1}}

\newcommand{\ic}{\mathbf{j}}

\newcommand{\br}[1]{({#1})}
\newcommand{\bigbr}[1]{\left({#1}\right)}
\newcommand{\sqb}[1]{\left[{#1}\right]}
\newcommand{\inv}[1]{{\left(#1\right)}^{-1}}

\newcommand{\vect}[1]{\mathbbold{#1}}
\newcommand{\vones}[1][]{\vect{1}_{#1}}

\DeclareSymbolFont{bbold}{U}{bbold}{m}{n}
\DeclareSymbolFontAlphabet{\mathbbold}{bbold}

\newcommand{\cx}{c}

\newcommand{\E}{\mathcal{E}}

\newcommand{\Gra}{\mathcal{G}}

\newcommand{\Vset}{v}
\newcommand{\Vmin}{v^{\min}}
\newcommand{\Vmax}{v^{\max}}

\newcommand{\V}{\mathcal{V}}

\newcommand{\is}{\mathbf{j}}

\newcommand{\Lap}{\mathsf{L}}

\newcommand{\Neb}[1]{\mathcal{N}\br{#1}}
\newcommand{\Par}[1]{\mathcal{P}\br{#1}}
\newcommand{\Chi}[1]{\mathcal{C}\br{#1}}

\newcommand{\sx}{s}
\newcommand{\Bnew}{\tilde{B}}
\newcommand{\Bnewtot}{\tilde{\mathbf{B}}}
\newcommand{\Bpar}{B}
\newcommand{\Gpar}{G}
\newcommand{\Btot}{\mathbf{B}}
\newcommand{\Gtot}{\mathbf{G}}
\newcommand{\conj}[1]{{\overline{#1}}}
\newcommand{\Vlog}{\gamma}
\newcommand{\R}{\mathbb{R}}

\definecolor{RED}{rgb}{0.6,0.,0.}
\definecolor{BLUE}{rgb}{0.,0.,0.6}
\definecolor{GREEN}{rgb}{0,0.5,0}
\definecolor{MALINA}{rgb}{0.6,0.,0.6}
\definecolor{YELLOW}{rgb}{0.8,0.8,0}
\newlength\tindent
\renewcommand{\indent}{\hspace*{\tindent}}

\newcommand{\tb}{\color{black}}
\newcommand{\djb}[1]{{\color{black} {#1}}}

\begin{document}
%
\title{High-voltage solution in radial power networks: Existence, properties  and equivalent algorithms}
%
%
%

\author{Krishnamurthy~Dvijotham,~\IEEEmembership{Member,~IEEE,}
        Enrique~Mallada,~\IEEEmembership{Member,~IEEE}
        and~John.W~Simpson-Porco,~\IEEEmembership{Member,~IEEE,}
\thanks{KD is with the Optimization and Control Group, Pacific Northwest National Laboratory, Richland, WA, 99354 USA e-mail: dvij@cs.washington.edu. KD was supported by the Control of Complex Systems Initiative at PNNL.}
\thanks{JS is with the department of Electrical and Computer Engineering, University of Waterloo, Canada and EM is with the department of Electrical and Computer Engineering, Johns Hopkins University}
\thanks{Manuscript received April 19, 2017; revised June 15, 2017.}}

\maketitle \footnote{This article has been accepted for publication in the IEEE Control Systems Letters journal \url{http://ieee-cssletters.dei.unipd.it/index.php}}


\begin{abstract}
The  AC power flow equations describe the steady-state behavior of the power grid. 
While many algorithms have been developed to compute solutions to the power flow equations, \djb{few} theoretical results are available characterizing when such solutions exist, or when these algorithms can be guaranteed to converge. 
{\tb In this paper, we derive necessary and sufficient conditions for the existence and uniqueness of a power flow solution in balanced radial distribution networks with homogeneous (uniform $R/X$ ratio) transmission lines.}
\djb{We study three distinct solution methods:  fixed point iterations}, convex relaxations, and energy functions \textemdash{} we show that the three algorithms successfully find a solution if and only if a solution exists.
\djb{Moreover, all three algorithms always find the unique high-voltage solution to the power flow equations, the existence of which we formally establish.}
{\tb \djb{ At this solution, we prove that (i) voltage magnitudes are increasing functions of the reactive power injections, (ii) the solution is a continuous function of the injections, and (iii) the solution is the last one to vanish as the system is loaded past the feasibility boundary.}}

\end{abstract}


%
\IEEEpeerreviewmaketitle

\section{Introduction}


\IEEEPARstart{T}{he} AC power flow equations are one of the most widely used modeling tools in power systems. They characterize the steady-state relationship between power injections at each bus and the voltage magnitudes and phase angle sthat are necessary to transmit power from generators to loads. 
\djb{They are embedded in every power system operations activity, including optimal power flow, state estimation,  security/stability assessment, and controller design~\cite{JM-JWB-JRB:08}.}

Since the power flow equations are nonlinear, solutions may not exist, and {\tb even when a solution exists, there may be multiple solutions}. The insolvability of these equations may indicate a system-wide instability, such as voltage collapse \cite{TVC:00}. Nonetheless, decades of experience show that there is typically a unique ``high-voltage'' solution \cite{DKM-DM-MN:16}. The high-voltage (small-current) solution \djb{is typically assumed to be the desired operating point for the network \cite{FD-MC-FB:11v-pnas}, although there are exceptions to this when the system is operated close to voltage collapse \cite{prada1998voltage}.} \djb{In general then, establishing the existence, uniqueness, and properties of this high-voltage solution is a prerequisite for static and dynamic network analysis.}


The second implication of power flow nonlinearity is that \textemdash{} even when solutions exist \textemdash{} finding them can be challenging. While many iterative algorithms (e.g., Newton-Raphson and Gauss-Seidel variants \cite{DM-DKM-KT:16}) are effective at solving power flow equations, they lack useful theoretical guarantees. In particular, when these algorithms fail to find a solution, it could be (i) due to an initialization issue, (ii) due to numerical instability, or (iii) that no solution exists to be found. This uncertainty introduces conservatism into system operation: when a perfectly good solution exists, but the solver fails to find it, the operator would mistakenly declare the case to be unsafe.


{\tb  These limitations of conventional nonlinear equation solvers have spurred the development of new conceptual frameworks for studying power flow equations \cite{DM-DKM-KT:16}. Each framework has a body of supporting theory and offers an algorithmic approach for solving the power flow equations. We consider three frameworks:

\begin{enumerate}\itemsep=1pt
\item[(i)] fixed point approaches, which exploit contraction/monotonicity properties of an operator to iteratively find a power flow solution \cite{SB-SZ:15, SY-HDN-KST:15, CW-AB-JYLB-MP:16, JWSP:17a, JWSP:17b, KD-SL-MC:15b}; \djb{here we propose and study a novel fixed point approach.}
\item[(ii)] convex relaxations, which cast the problem of finding a power flow solution as a convex optimization problem; a convex super-set containing all power flow solutions is defined, and minimization of a carefully chosen function over this set yields the power flow solution \cite{LG-NL-UT-SL:12,SHL:14a,JL-DT-BZ:14,SHL:14b}.
\item[(iii)] energy functions, where the stable power flow solution is characterized as a local minimizer of an appropriately defined scalar function; the function can then be minimized by gradient descent to find the solution \cite{DjEnergyFun}.\footnote{\djb{Note that in this paper we use energy function in a generalized sense so that it may not actually correspond to a Lyapunov function of the dynamics. However, we still abuse terminology to refer to minimizers of the energy function as ``stable'' solutions.}}
\end{enumerate}

}

%
{\tb For a given set of power injections, we show that there are only two possibilities. Either
\begin{enumerate}
\item[1)] there are no power flow solutions, or
\item[2)] there is a unique ``high voltage'' power flow solution\footnote{High-voltage meaning that the voltage magnitude at each bus is higher than the corresponding voltage magnitude of any other solution.}, and all three approaches described above find this solution.
\end{enumerate}}

{\tb The properties we establish (existence of a unique high-voltage/voltage-regular/stable solution) have been conjectured to be true for distribution networks. \djb{To the best of our knowledge though, formal proofs of these properties have never been presented. Rigorous proofs are difficult and require special assumptions on the distribution system.} In the next subsection, we summarize these assumptions and precisely define the contributions of this paper.

\subsection{Assumptions made and contributions of the paper}
\label{Sec:Contributions}
We make several assumptions on the distribution network:\\
 (i) \emph{Simplified line model:}  the shunt elements typically present in the $\Pi$-model are neglected, along with voltage control mechanisms like capacitor banks and tap-changing transformers; all distribution lines are modeled as series impedances. \\
(ii) \emph{Balanced operation:} the three phases of the system are balanced, allowing for a single-phase representation.\\
(iii) \djb{\emph{Constant $r/x$ ratios of lines’ longitudinal impedance parameters:}} all distribution lines have an equal ratio of
resistance to reactance.\\
(iv) \emph{Strictly positive voltage magnitudes:} we consider power flow solutions with strictly positive voltage magnitudes at every bus.\\
While \djb{these assumptions limit the applicability of our results to practical distribution systems, we leverage them to derive equally strong results:}\\
    (i) \emph{Existence conditions:} We establish necessary and sufficient conditions for the existence of solutions. This is \djb{in contrast to previous works} that only establish sufficient \cite{SY-HDN-KST:15}\cite{JWSP:17a}\cite{JWSP:17b} or necessary \cite{molzahn2012sufficient} conditions.\\
    (ii) \emph{Connections between approaches:} We \djb{establish} connections between the three approaches (fixed point, relaxation, energy function) and show that all three approaches either find the \emph{same high-voltage solution} (if one exists) or fail (if no power flow solution exists). These guarantees are stronger than the guarantees established in previous work  on computing power flow solutions \cite{DjEnergyFun,jbar2006,SB-SZ:15,SY-HDN-KST:15}.\\
    (iii) \emph{Properties of the power flow solution:} We establish several properties of the power flow solution (voltage regularity, stability and continuity). While these properties have been conjectured in previous work ~\cite{HDC-MEB:90, KNM-HDC:00}, we provide rigorous proofs of these from first principles.
We envision that the strong results established under the special assumptions in this paper will form the basis of further studies that will extend the results to be applicable to practical distribution networks.}

The rest of the paper is organized as follows. Section \ref{Sec:ACPF} \djb{introduces the power flow equations}. Section \ref{Sec:Algorithms} describes three \djb{solution algorithms} based on  fixed-point iterations, convex relaxations, and energy functions.
Our main contributions are developed in Section \ref{Sec:Main} where we show the equivalence of the three representations and describe additional properties of the high-voltage solution. Section \ref{Sec:Conclusions} \djb{concludes the paper.}

\section{AC power flow equations}\label{Sec:ACPF}
{\tb We start with notation that is used in this paper: \\
$\mathbb{C}$: Set of complex numbers, 
$\is$: $\sqrt{-1}$ \\
$\mathrm{arg}(x)$: Phase of the complex number x \\
$|x|$: Magnitude of the complex number x \\
$\log(x)$: The vector with entries $\log(x_i)$ for any $x \in \R^n$ \\
$\exp(x)$: The vector with entries $\exp(x_i)$ for any $x \in \R^n$ \\
$[a,b]$: For $a,b\in\R^n$ with $a\leq b$ (componentwise), this denotes $\{x\in \R^n: a \leq x \leq b\}$ \\
Z-matrix: Square matrix with non-positive off-diagonal entries
$e_i$: Vector with all entries $0$ except the $i$-th entry (equal to $1$) \\
$[n]$: $\{1,2,\ldots,n\}$ for any natural number $n$ \\
\djb{$\vones[n]$: Vector in $\R^n$ with all entries equal to $1$}
}

We will focus exclusively on radial (tree) \djb{AC} power networks in  steady-state. The grid topology is that of a rooted \djb{oriented} tree $\Gra=\br{\V,\E}$ where $\V = \{0,\ldots,n\}$ is the set of buses and (by convention) the lines $\br{i,k}\in\E$ are oriented away from the substation bus $0$. Each line $\br{i,k}\in\E$ in the network has an associated complex admittance $Y_{ik} = Y_{ki} =\Gpar_{ik}-\is \Bpar_{ik}$, where $G_{ik} \geq 0$ and $B_{ik} > 0$.
{\tb We let $\Neb{i} = \{k: (i,k)\in \E \text{ or } (k,i) \in \E\}$ denote the set of neighbors of bus $i$.}

At every bus $i \in \V$, we denote the voltage phasor by $V_i \in \mathbb{C}$, the squared voltage magnitude by $\Vset_i = |V_i|^2$, the voltage phase by {$\theta_i = \mathrm{arg}(V_i)$}, (net) active power injection by  $p_i$, and (net) reactive power injection by $q_i$. Bus $0$ is interpreted as a substation, and taken to be the reference (slack) bus with voltage phasor $1\expb{\is 0}$ per unit. Every other bus $i \in \{1,\ldots,n\}$ is a PQ bus with $p_i, q_i$ specified and $\Vset_i, \theta_i$ to be determined.
%
%
Under assumptions (i) and (ii) (Section \ref{Sec:Contributions}), the power balance equation at bus $i\in\V$ is
\begin{equation}\label{eq:PFa}
\begin{aligned}
p_i + \ic q_i &= V_i\sum_{k\in\Neb{i}}\conj{Y_{ik}}\br{\conj{V_i}-\conj{V_k}}\,,
\end{aligned}
\end{equation}
where $\conj{x}$ denotes the complex conjugate of $x$.
The voltage phase difference across line $\br{i,k}\in \E$ is denoted by $\phi_{ik}=\theta_i-\theta_k$.
Substituting $V_i={\sqrt{\Vset_i}}\expb{\ic\theta_i}$ and taking real and imaginary parts of \eqref{eq:PFa}, a simple calculation shows that
\begin{subequations}\label{eq:PFrealreac}
\begin{align}
p_i &= \Gtot_{i}\Vset_i + \sum_{k \in \Neb{i}}\nolimits \br{\Bpar_{ik}s_{ik} - \Gpar_{ik}c_{ik}}
\label{eq:PFreal} \\
q_i &= \Btot_{i}\Vset_i + \sum_{k \in \Neb{i}}\nolimits \br{-\Gpar_{ik}s_{ik}-\Bpar_{ik}c_{ik}}
\label{eq:PFreac}
\end{align}
\end{subequations}
for each $i \in \V\setminus\{0\}$, where
\begin{equation*}
\Gtot_i:= \sum_{k\in\Neb{i}} \nolimits \Gpar_{ik} \qquad \Btot_i:= \sum_{k\in\Neb{i}}\nolimits \Bpar_{ik}
\end{equation*}
denote the total conductance/susceptance incident to bus $i$, and we have used the simplifying variables
\begin{align}\label{eq:newVars}
\sx_{ik}:=\sqrt{\Vset_i\Vset_{k}}\sin\br{\phi_{ik}},\quad
\cx_{ik} :=\sqrt{\Vset_i\Vset_{k}}\cos\br{\phi_{ik}}\,.
\end{align}
\djb{The power flow equations can be written using only the variables}  $\cx_{ik},\sx_{ik}$ and $\Vset_i$, with the additional constraints that
\begin{equation}\label{Eq:Constraints}
\cx_{ik}^2+\sx_{ik}^2=\Vset_i\Vset_{k}\,, \qquad (i,k)\in\E\,.
\end{equation}
The phase differences $\phi=\{\phi_{ik},\br{i,k}\in\E\}$ can be recovered uniquely (modulo $2\pi$) via $\sin\br{\phi_{ik}}=\frac{\sx_{ik}}{\sqrt{\Vset_i\Vset_{k}}}$.\footnote{\tb As mentioned before, we assume $v_i>0$ for all $i\in \V$.
Having $v_i=0$ for some bus $i$ necessarily implies that no power is being transferred through the lines adjacent to it which implies that no load can exist at that bus. }
Since the network is a tree, $|\E| = n$, and once $\theta_0=0$ is fixed there is a one-to-one mapping between $\phi$ and $\theta=\{\theta_{i}\}_{i=1}^{n}$, which allows the phase angles $\theta$ to also be uniquely recovered. {\tb Since all lines are homogeneous (assumption (iv) from Section \ref{Sec:Contributions})},  $\Gpar_{ik}/\Bpar_{ik} = \kappa$ some $\kappa > 0$ and for all $(i,k)\in \E$. 
We create a new system of equations by subtracting $\kappa$ times \eqref{eq:PFreac} from \eqref{eq:PFreal}, and adding $\kappa$ times \eqref{eq:PFreal} to \eqref{eq:PFreac}, to obtain
\begin{subequations}\label{eq:PFnew}
\begin{align}
\tilde{p}_i &:= p_i - \kappa q_i = \sum_{k \in \Neb{i}}\nolimits \Bnew_{ik} \sx_{ik}\,, \label{eq:PFrealeasy} \\
\tilde{q}_i &:= q_i + \kappa p_i = \Bnewtot_i\Vset_i - \sum_{k \in \Neb{i}}\nolimits \Bnew_{ik}\cx_{ik}\,,
\label{eq:PFreaceasy}
\end{align}
\end{subequations}
\djb{where $\Bnew_{ik} := (1+\kappa)\Bpar_{ik}$ and $\Bnewtot_{i} := \Btot_{i}(1+\kappa)$}. 
For notational simplicity, going forward we will drop the $\tilde{\cdot}$ and use $\Bpar_i,\Btot_i,p_i,q_i$ instead of $\Bnew_i,\Bnewtot_i,\tilde{p}_i,\tilde{q}_i$. The equation \eqref{eq:PFrealeasy} is a \djb{square, full-rank} linear system in the variables $s = \{s_{ik}\}$. Thus  \eqref{eq:PFrealeasy} can be uniquely solved for the flows $\sx$ as a function of $p$, and we denote this solution by $\sx=\sx\br{p}$. The power flow equations \eqref{eq:PFnew} with the constraints \eqref{Eq:Constraints} then simplify to
\begin{subequations}\label{eq:PFfinal}
\begin{align}
& \Btot_i\Vset_i-\sum_{k \in \Neb{i}}\nolimits \Bpar_{ik}\cx_{ik} = q_i\,,\quad &\forall &i \in [n]  \\
& \sx_{ik}\br{p}^2+\cx_{ik}^2= \Vset_i\Vset_{k}\,,\quad  &\forall &(i,k)\in \E\,.  
\end{align}
\end{subequations}
 We will sometimes drop the dependence $\sx_{ik}\br{p}$ for brevity and simply write $\sx_{ik}$.
\djb{Let $\Lap = \Lap^{\sf T} \in \mathbb{R}^{(n+1)\times(n+1)}$ be the \emph{Laplacian matrix} of the (undirected) graph, with entries $\Lap_{ik} = -B_{ik}$ and $\Lap_{ii} = \Btot_{i}$.
We may write $\Lap$ as the block matrix
$$
\Lap = \begin{pmatrix}
\Lap_{00} & \Lap_{0\mathrm{r}}\\
\Lap_{\mathrm{r}0} & \Lap_{\rm red}
\end{pmatrix}
$$
where the first row/column corresponds to the substation and $\Lap_{\rm red} \in \mathbb{R}^{n \times n}$ is the \emph{reduced Laplacian matrix}. We let $\Btot = \mathrm{diag}(\Btot_i)_{i=0}^n$ be the diagonal matrix with entries $\Btot_i$, and $\Btot_{\rm red} = \mathrm{diag}(\Btot_i)_{i=1}^n$ be the corresponding reduced matrix.} 

\section{Solving the power flow equations}
\label{Sec:Algorithms}
{\tb We \djb{now} describe three approaches for computing solutions to the power flow equations \eqref{eq:PFfinal}. } 
\subsection{Convex relaxation approach}
We start with the power flow equations \eqref{eq:PFfinal} and relax both equations to inequalities:
\begin{equation}\label{Eq:Relaxation1}
\Btot_i \Vset_i-\sum_{k \in  \Neb{i}} B_{ik}c_{ik}\leq q_i \,,\quad
\br{s_{ik}(p)}^2+c_{ik}^2 \leq v_iv_k\,,
\end{equation}
for all $i \in \sqb{n}$, and  $\br{i,k}\in\E$, respectively. {\tb The inequalities \eqref{Eq:Relaxation1} are equivalent to the constraints of the  second-order conic relaxation described in \cite{jbar2006}. While \cite{jbar2006} was the first to formulate this conic relaxation, no guarantees were established on when this approach provably finds a power flow solution. In the context of the optimal power flow problem, \cite{SHL:14b} describes several results proving tightness of the convex relaxation, that is, that the global optimum of the non-convex OPF problem can be computed by solving a convex relaxation. Here, we are concerned with simply \emph{solving the power flow equations}, and indeed, the results of \cite{SHL:14b} imply that the power flow equations for radial networks can be solved by solving a convex optimization problem. In this paper, we use the convex relaxation formulation as a tool to establish several properties of the power flow solution for radial networks.}

Lemma \ref{lem:Lemb} shows that the relaxation \eqref{Eq:Relaxation1} is feasible if and only if there exists a voltage vector $\Vset$ such that
\begin{align}
\Btot_i \Vset_i-\sum_{k \in  \Neb{i}} B_{ik}\sqrt{\Vset_i\Vset_k-s_{ik}^2}\leq q_i\,, \quad \forall i \in \sqb{n}\,.
\label{eq:PFred}
\end{align}
We will also refer to the equality form of the above constraint
\begin{align}
\Btot_i \Vset_i-\sum_{k \in  \Neb{i}} B_{ik}\sqrt{\Vset_i\Vset_k-s_{ik}^2}= q_i\,, \quad \forall i \in \sqb{n}\,.
\label{eq:PFredeq}
\end{align}

{\tb We propose computing a power flow solution by maximizing a weighted linear combination of log-voltage magnitudes:
%
\begin{equation}\label{eq:PFRelaxOpt}
\begin{aligned}
& \underset{v \in \mathbb{R}^n_{>0}}{\text{maximize}}
& & \sum_{i=1}^n \nolimits w_i \log\br{\Vset_i} \quad
\text{ subject to } \eqref{eq:PFred}
\end{aligned}
\end{equation}
were $w_i>0$ are arbitrary positive weights.}
%
%
\begin{definition}\label{def:RelaxSuccess}
The relaxation approach \emph{succeeds} if \eqref{eq:PFRelaxOpt} is feasible and the optimal solution satisfies \eqref{eq:PFredeq}. In this case, a solution to \eqref{eq:PFfinal} can be computed by taking the optimal solution $\Vset^\star$ and defining $\sx^\star_{ik}:=\sx_{ik}(p)$, $\cx^\star_{ik}:=\sqrt{\Vset_i\Vset_k-\br{\sx^\star_{ik}}^2}$.
Otherwise, we say the approach \emph{fails}.
\end{definition}

\subsection{Energy function approach}

Power flow solutions can also be interpreted as stationary points of the \emph{energy function} introduced in \cite{varaiya1985direct, DjEnergyFun}. \djb{In the context of transient stability in lossless transmission networks, it is known that the energy function is a Lyapunov function of the power system swing dynamics, and that any power flow solution at which the energy function is locally convex is asymptotically stable. We do not pursue such a stability analysis here, and instead use the energy function in a broader sense \textemdash{} it is simply a function whose stationary points correspond to power flow solutions. We will nonetheless refer to power flow solutions at which the energy function is locally convex (i.e., its Hessian is positive definite) as ``stable'' solutions. Our main results will establish that such a solution indeed exists, and is unique.
Our approach here builds on \cite{DjEnergyFun}.}
%
The energy function {\tb $\jmap{E}{\real^{n}_{>0} \times \real^n}{\real}$} is defined as
\begin{align}
E\br{\Vset,\theta}&:=\sum_{\br{i,k}\in \E}B_{ik}\br{\Vset_i+\Vset_k-2\sqrt{\Vset_i\Vset_k}\cos\br{\theta_i-\theta_k}}\nonumber\\&\quad +\sum_{i=1}^n\nolimits p_i \theta_i + \sum_{i=1}^{n}\nolimits q_i \log\br{\Vset_i}/2\,.  \label{eq:EFdef}
\end{align}
One may check directly that
\begin{align}
\frac{\partial E}{\partial \log\br{\Vset}}=0 \equiv \eqref{eq:PFreaceasy}\,, \qquad
\frac{\partial E}{\partial \theta}=0 \equiv \eqref{eq:PFrealeasy}\,, \label{eq:Estat}
\end{align}
which shows that power flow solutions are simply stationary points of the energy function. 
\begin{theorem}[\cite{DjEnergyFun}]
The energy function $E$ is jointly convex in $\br{\log\br{\Vset},\theta}$ provided that
\begin{subequations}\label{eq:EnergyFunCondvcs1}
\begin{align}
\sum_{i=1}^n \Btot_i \sqb{2\Vset_i}_{i} -\sum_{\br{i,k}\in \E} \frac{\Bpar_{ik}\sqrt{\Vset_i\Vset_{k}}}{\cx_{ik}} \sqb{\begin{pmatrix} 1 & 1 \\
1 & 1
\end{pmatrix}}_{i,k}\succeq 0 \label{eq:EnergyFunCondvcsa}\\
\cx_{ik}=\sqrt{\Vset_i\Vset_k}\cos\br{\theta_i-\theta_k}>0 
\end{align}
\end{subequations}
where \[\sqb{a}_i=a\mathrm{e}_i\mathrm{e}_i^{\sf T},\sqb{\begin{pmatrix} a & b \\
b & c
\end{pmatrix}}_{i,k} = a\mathrm{e}_i\mathrm{e}_i^{\sf T} + b(\mathrm{e}_i\mathrm{e}_k^{\sf T}+\mathrm{e}_{k}\mathrm{e}_i^{\sf T}) + c\mathrm{e}_k\mathrm{e}_k^{\sf T}\]
The constraints \eqref{eq:EnergyFunCondvcs1} define a convex set in $\br{\logb{\Vset},\theta}$.\footnote{\djb{The condition $c_{ik}>0$ is equivalent to requiring that the phase differences between neighboring buses are smaller than $\frac{\pi}{2}$. In Theorem \ref{thm:MainExist}, it is shown that this is always the case for the high-voltage power flow solution.}}
\end{theorem}

The power flow equations \eqref{eq:PFnew} can therefore be solved by solving the following optimization problem:
\begin{equation}\label{eq:Efunconvex}
\underset{v,\theta}{\text{minimize}} \,\, E(v,\theta) \qquad \text{subject to}
\,\,\eqref{eq:EnergyFunCondvcs1}\,.
\end{equation}
{\tb Convexity allows us to draw strong conclusions:} if there is a power flow solution in the interior of the domain of the set \eqref{eq:EnergyFunCondvcs1}, it is the unique power flow solution in the set \eqref{eq:EnergyFunCondvcs1}. Conversely, if the optimizer of \eqref{eq:Efunconvex} is not a solution, there are no solutions to the PF equations in the interior of the set \eqref{eq:EnergyFunCondvcs1}.

\begin{definition}\label{def:EFSuccess}
The energy function approach \emph{succeeds} if there exists an optimizer of \eqref{eq:Efunconvex} that satisfies \eqref{eq:Estat}. In this case, a solution to \eqref{eq:PFfinal} can be computed by taking the optimizer $\br{\Vset^\star,\theta^\star}$ and defining
$\sx^\star_{ik}=\sqrt{\Vset^\star_i\Vset^\star_k}\sin\br{\theta^\star_i-\theta^\star_k}$,
$\cx^\star_{ik}=\sqrt{\Vset^\star_i\Vset^\star_k}\cos\br{\theta^\star_i-\theta^\star_k}$.
Otherwise, we say the approach \emph{fails}.
\end{definition}

\subsection{Fixed-point approach}
{\tb Simply by rearranging, the power flow equations \eqref{eq:PFredeq} can be rewritten as fixed-point equations:}
\begin{align*}
\Vset_i= g_i(v) := \frac{q_i}{\Btot_i}+\sum_{k \in \Neb{i}}\frac{B_{ik}}{\Btot_i}\sqrt{\Vset_i\Vset_k-s_{ik}^2}\,, \quad \forall i \in \sqb{n}\,.
\end{align*}
This system can be written in vector form as $\Vset=g\br{\Vset}$.
{\tb We can therefore define a fixed-point iteration
\begin{align}
\Vset^{(i+1)} \gets g\br{\Vset^{(i)}}\,, \quad \djb{\Vset^{(0)} = \Vset^{\max} := \vones[n] + 2\inv{\Lap_{\rm red}}q\,.}
\label{eq:PFfp}
\end{align}
Lemma \ref{lem:PFupper} shows that $\Vset^{\max}$ is an upper bound on the voltage magnitudes of any power flow solution.}

\begin{definition}\label{def:FPSuccess}
The fixed-point approach \emph{succeeds} if the iteration  \eqref{eq:PFfp} converges \djb{to a point $v^\star \in \mathbb{R}_{>0}^n$} satisfying $g\br{\Vset}=\Vset$. In this case, a solution to \eqref{eq:PFfinal} can be computed by taking the fixed point  $\Vset^\star$ and defining $\sx^\star_{ik}=\sx_{ik}(p)$, $\cx^\star_{ik}=({\Vset_i\Vset_k-\br{\sx^\star_{ik}}^2})^{\frac{1}{2}}$.
Otherwise, we say the approach \emph{fails}.
\end{definition} 
\section{Theoretical results on PF approaches and properties of the PF solution}
\label{Sec:Main}

For each approach presented in Section \ref{Sec:Algorithms}, it is of interest to understand when the approach succeeds \djb{and when it fails}. The following theorem addresses these questions.
\begin{theorem}[Equivalent Power Flow Approaches]\label{thm:MainExist}
The following two statements are equivalent:
\begin{enumerate}
\item[(i)] the power flow equations \eqref{eq:PFfinal} have a solution;
\item[(ii)] the approaches \eqref{eq:PFRelaxOpt}, \eqref{eq:Efunconvex}, and \eqref{eq:PFfp} succeed.
\end{enumerate}
If either of these equivalent statements are true, all three approaches compute the same power flow solution.
\end{theorem}
\begin{proof}
We begin by proving (i) $\Longrightarrow$ (ii). \djb{Suppose that a power flow solution exists. The reduced power flow equations  \eqref{eq:PFredeq} have a solution, and hence $\exists \Vset^a$ such that $g\br{\Vset^a}=\Vset^a$ and hence $\exists \Vset^a$ such that $g\br{\Vset^a}\geq \Vset^a$.  Lemma \ref{lem:PFupper} shows that $\Vset^a \leq \Vset^{\max}$ and $g\br{\Vset^{\max}}\leq \Vset^{\max}$. Note that the map $g$ is monotone, since each component of $g$ is a non-decreasing function of each $\Vset_i$. Further, we have for any $\Vset\in [\Vset^{a},\Vset^{\max}]$ that 
$g\br{\Vset}\geq g\br{\Vset^{a}}\geq \Vset^{a}$, and that $g\br{\Vset}\leq g\br{\Vset^{\max}} \leq \Vset^{\max}$.  We invoke Theorem \ref{thm:Tarski} (Appendix \ref{sec:App}) with $a=\Vset^{a},b=\Vset^{\max}$, and $F=g$, establishing that the fixed-point iteration \eqref{eq:PFfp} converges to a fixed point and in fact converges to a maximal fixed point $\Vset^{\star}$, such that $\Vset^{\star}\geq \Vset$ for every \[\Vset \in \{\Vset \text{ s.t } \Vset^a\leq \Vset \leq \Vset^{\max}, g\br{\Vset}\geq \Vset\}.\] }
\paragraph{Proof that the fixed-point approach succeeds}  \djb{The above argument proves that the fixed point approach succeeds since \eqref{eq:PFfp} converges to the PF solution $\Vset^\star$. }
\paragraph{Proof that the convex relaxation approach succeeds}

\djb{By scaling each constraint in \eqref{eq:PFRelaxOpt} by $\Btot_i$ and rearranging, the feasible set of \eqref{eq:PFRelaxOpt} can be written as $g\br{\Vset}\geq \Vset$. Since there is a power flow solution, \eqref{eq:PFRelaxOpt} is feasible. Given any feasible solution of $\Vset^\prime$ of \eqref{eq:PFRelaxOpt}, the argument from the beginning of this proof with $\Vset^a=\Vset^\prime$ establishes that $\Vset^\star\geq \Vset^\prime$.} Since the objective of \eqref{eq:PFRelaxOpt} is strictly increasing in each component of $\Vset$, $\Vset^\star$ must be the optimal solution since $\Vset^\star \geq \Vset^\prime$ for every  feasible solution $\Vset^\prime$. Thus, if there is a power flow solution, \eqref{eq:PFRelaxOpt} has a unique optimizer that satisfies \eqref{eq:PFredeq} and hence the relaxation approach succeeds.
\paragraph{Proof that the energy function approach succeeds} 

{\tb Define a new mapping $\jmap{\tilde{g}}{\real^n}{\real^n}$ as $\tilde{g}(\Vlog)=\mathbf{B}_{\rm red}(\expb{\Vlog}-g\br{\expb{\Vlog}})$, where $\gamma \in \R^n$.
Lemma \ref{lem:EF} shows that \djb{$\frac{\partial \tilde{g}}{\partial \gamma}$ is symmetric and that} the convexity conditions \eqref{eq:EnergyFunCondvcs1} are equivalent to the condition $\frac{\partial \tilde{g}}{\partial \gamma} \succeq 0$ on the Jacobian of $\tilde{g}$. Further, since \djb{$\tilde{g}_i\br{\Vlog}=\Btot_i\br{ \expb{\Vlog_i} - g_i\br{\expb{\Vlog_i}}}$} and $g_i$ is increasing in \ each component of $\Vlog$, $\frac{\partial \tilde{g}_i}{\partial \Vlog_k}\leq 0$ for each $k \neq i$. Thus, the matrix $\frac{\partial \tilde{g}}{\partial \Vlog}$ is a Z-matrix.} We rewrite \eqref{eq:PFRelaxOpt} as an optimization problem where the decision variables are $\gamma=\log\br{\Vset}$ as follows:
\begin{align*}
\underset{\Vlog}{\text{maximize}} \,\, \sum_i w_i \gamma_i \qquad \text{subject to}
\,\,\quad \tilde{g}\br{\gamma} \leq q\,,
\end{align*}
and the unique optimal solution (from part (b)) is $\log\br{\Vset^\star}$. Writing the KKT conditions for this problem establishes existence of a vector $\lambda \geq 0$ such that $\frac{\partial \tilde{g}}{\partial \gamma}\lambda = w.$
By Lemma \ref{thm:Z}, this implies that $\frac{\partial \tilde{g}}{\partial \gamma}$ must be positive definite at the optimal solution. Define $\gamma^\star=\log\br{\Vset^\star}$ and $\theta^\star$ such that $\theta^\star_i-\theta^\star_k=\arcsin\br{{s_{ik}}/{\sqrt{\Vset^\star_i\Vset^\star_k}}}$, with $\theta_0^\star = 0$. The positive definiteness of $\frac{\partial \tilde{g}}{\partial \gamma}$ at $\gamma = \gamma^\star$ shows that the energy function optimization problem \eqref{eq:Efunconvex} has a stationary point $(\gamma^\star,\theta^\star)$ in the interior of the feasible set. Since the energy function is strongly convex in the interior, this stationary point is an isolated local minimum. Since \eqref{eq:Efunconvex} is a convex problem, this stationary point is the unique global optimum.

\emph{Proof of converse (ii) $\Longrightarrow$ (i):} Conversely, if there are no power flow solutions, then \eqref{eq:PFRelaxOpt} must be infeasible.  If this was not true, then $\exists \Vset^a$ \djb{such that $g\br{\Vset^a}\geq \Vset^a$. Using Lemma \ref{lem:PFupper}, $\exists \Vset^{\max}$ such that $g\br{\Vset^{\max}}\leq \Vset^{\max}$, $\Vset^a\leq \Vset^{\max}$ so that $\Vset^a \leq g\br{\Vset^a}\leq g\br{\Vset^{\max}} \leq \Vset^{\max} $. Hence by Theorem \ref{thm:Tarski}, there is a fixed point $g\br{\Vset}=\Vset$ and hence a  power flow solution exists (which is a contradiction)}. Thus, the convex relaxation must be infeasible and the convex relaxation approach cannot succeed. Since $g$ has no fixed points (as there are no power flow solutions) the iterative procedure cannot converge to a fixed point and hence the fixed-point approach fails. Finally, the energy function approach succeeds only if the optimum of \eqref{eq:Efunconvex} is a stationary point of the energy function. Since stationary points are power flow solutions, there are no stationary points and hence the energy function approach cannot succeed.
\end{proof}

{\tb Theorem \ref{thm:MainExist} says that the three approaches from Section \ref{Sec:Algorithms} all find a power flow solution if and only if a power flow solution exists, and in fact, they all find the same solution. If any of them fail, then no power flow solution exists. 
\djb{To state our next result,} which establishes several desirable properties of the power flow solution, we define some notation. Let}
$$
\mathcal{S}:=\{\br{p,q} \in \mathbb{R}^{n}\times \mathbb{R}^n \text{ s.t. } \eqref{eq:PFRelaxOpt} \text{ is feasible}\}\,
$$
be the feasible set of injections. For each $\br{p,q}\in \mathcal{S}$, let $\Vset\br{p,q}$ be the unique optimal solution of \eqref{eq:PFRelaxOpt}, and \djb{define $\theta\br{p,q} \in \mathbb{R}^{n+1}$ as the unique solution to}
\djb{
$$
\theta_i\!-\!\theta_k=\arcsin\bigbr{\frac{s_{ik}\br{p}}{\sqrt{\Vset_i\br{p,q},\Vset_k\br{p,q}}}}\,\quad  \forall \br{i,k}\in \E\, ,\, \theta_0=0
$$.}

\begin{theorem}[{\tb Properties of High-Voltage Solution}]\label{thm:Properties}
The following statements hold:
\begin{enumerate}[wide, labelwidth=!, labelindent=0pt]
\item[(i)] $\mathcal{S}$ is a convex set and the map $\br{p,q} \mapsto \br{\Vset\br{p,q},\theta\br{p,q}}$
is continuous on $\mathcal{S}$;
\item[(ii)] for any $\br{p,q}\not\in\mathcal{S}$, there are no solutions to \eqref{eq:PFfinal};
\item[(iii)] for any $\br{p,q} \in \mathcal{S}$, $\br{\Vset\br{p,q},\theta\br{p,q}}$ {\tb  is the \emph{unique solution} to the power flow equations that satisfies the following properties}:
\begin{enumerate}
\item[]\emph{\textbf{High-voltage:}} $\Vset\br{p,q}\geq \Vset^\prime$ for every solution $\Vset^\prime$ of \eqref{eq:PFredeq}.
\item[]\emph{\textbf{Stability:}} $\br{\Vset\br{p,q},\theta\br{p,q}}$ lies in the domain of convexity of the energy function (i.e., satisfies \eqref{eq:EnergyFunCondvcs1});
\item[]\emph{\textbf{Voltage-regularity:}} the matrix $\frac{\partial \Vset}{\partial q}$ is element-wise positive at the solution.
\end{enumerate}
\end{enumerate}
\end{theorem}
\begin{proof}


\textbf{(i):} Since $\sqrt{xy-z^2}$ is a concave function in $\br{x,y,z}$, the feasible set of \eqref{eq:PFRelaxOpt} is jointly convex in $\br{\Vset,p,q}$. The injection set $\mathcal{S}$ is simply the projection of this set onto the last two arguments, and is therefore convex. Further, \eqref{eq:PFRelaxOpt} has a strongly concave objective and hence a unique optimum. It follows from standard results in parametric convex optimization \cite[Section 6.1]{IEORNotes} that $\Vset\br{p,q}$ is a continuous map, and hence so is $\theta\br{p,q}$. 

\textbf{(ii):} If $\br{p,q}\not\in\mathcal{S}$, \eqref{eq:PFRelaxOpt} is infeasible and hence there are no power flow solutions. 

{\tb
\textbf{(iii):} The first two properties of the specified solution follow immediately from the proof of Theorem \ref{thm:MainExist}. For the third, that same proof showed that 
\begin{equation}\label{Eq:Jacobians}
\frac{\partial \tilde{g}}{\partial \gamma} = \frac{\partial q}{\partial \log(v)} = \frac{\partial q}{\partial v}\left(\frac{\partial \log(v)}{\partial v}\right)^{-1} = \frac{\partial q}{\partial v}\,\mathrm{diag}(v)
\end{equation}
evaluated at the solution is a positive-definite Z-matrix. \djb{Since the network is connected, this matrix is irreducible.} The result follows by applying Lemma \ref{thm:Z}.
}
\end{proof}

\section{Conclusions}
\label{Sec:Conclusions}
{\tb \djb{We have developed and analyzed several approaches to solving the power flow equations for balanced radial networks with transmission lines characterized by homogeneous ratios of longitudinal impedance parameters}. We showed these approaches are equivalent: they all either succeed and find the high-voltage power flow solution, or they all fail and no power flow solution exists to be found. In the former case, we established several desirable properties of the power flow solution found by each method. While some of these approaches were known, this is the first paper to systematically study \djb{the connections between these approaches}. 

These results form a solid foundation for further investigation. \djb{Immediate future work will study relaxing the assumptions from Section \ref{Sec:Contributions}.} Establishing analogous results for meshed power networks remains an open problem. 

\section*{Acknowledgment}
KD would like to thank the Control of Complex Systems Initiative at PNNL for supporting this work. The authors thank Florian Dorfler for encouraging their work on this problem and Florian Dorfler and Saverio Bolognani for helpful comments and feedback on this work. 

\bibliographystyle{IEEEtran}
\bibliography{IEEEabrv,alias,johnbib,Ref}
\appendices
\section{Supporting lemmas}\label{sec:App}
{\tb 
\begin{lemma}\label{lem:Lemb}
There exist $\Vset \in \mathbb{R}^n_{>0}$ and $c \in \mathbb{R}^{|\E|}$ satisfying \eqref{Eq:Relaxation1}
if and only if there exists $\Vset \in \mathbb{R}^n_{>0}$ satisfying
\begin{align}
& \Btot_i \Vset_i-\sum_{k \in \Neb{i}} B_{ik}\sqrt{\Vset_i\Vset_k-s_{ik}^2} \leq q_i\,, \quad \forall i \in \sqb{n}\,.
\label{eq:Relb}
\end{align}
\end{lemma}
\begin{proof}
The second constraint of \eqref{Eq:Relaxation1} implies that
$$
-\sqrt{\Vset_i\Vset_k-s_{ik}^2}\leq c_{ik}\leq \sqrt{\Vset_i\Vset_k-s_{ik}^2}\,.
$$
If \eqref{Eq:Relaxation1} does not hold for $c_{ik}=\sqrt{\Vset_i\Vset_k-s_{ik}^2}$, it cannot hold for any other value of $c_{ik}$. We can therefore rewrite the relaxation as $\Btot_i \Vset_i-\sum_{k \in  \Neb{i}} B_{ik}\sqrt{\Vset_i\Vset_k-s_{ik}^2}\leq q_i\,$.
\end{proof}
\begin{lemma}\label{lem:PFupper}
Any solution to $g\br{\Vset}\geq \Vset$ satisfies \djb{$\Vset \leq \Vset^{\max}$.}
Further, we have $g\br{\Vset^{\max}}\leq \Vset^{\max}$.
\end{lemma}
\begin{proof}
The proof relies on the inequality
$\sqrt{\Vset_i\Vset_k-s_{ik}^2} \leq \sqrt{\Vset_i\Vset_k}\leq \frac{\Vset_i+\Vset_k}{2}$.
We apply this to $g$ to obtain:
$$
\begin{aligned}
g_i\br{\Vset}-\Vset_i &\leq \frac{q_i}{\Btot_i}+\sum_{k \in \Neb{i}} \frac{B_{ik}}{\Btot_i}\frac{\Vset_i+\Vset_k}{2}-\Vset_i\\ 
&= \frac{q_i + \djb{\sum_{k\in\Neb{i}}} \frac{B_{ik}}{2}(\Vset_k-\Vset_i)}{\Btot_i}
\end{aligned}
$$
\djb{In vector notation, this inequality can be written as
\begin{align}
g\br{v}-v &\leq  \inv{\mathbf{B}_{\rm red}}\left(q-\frac{1}{2}(\Lap_{\rm red}\Vset+\Lap_{\mathrm{r}0})\right)\nonumber \\
&= \frac{1}{2}\inv{\mathbf{B}_{\rm red}}{\Lap_{\rm red}}\br{\Vset^{\max}-\Vset} \label{eq:tmpLem2}
\end{align}
where $A = \Btot_{\rm red}^{-1}\Lap_{\rm red}/2$ and we used $-\Lap_{\rm red}^{-1}\Lap_{\mathrm{r}0} = \vones[n]$ and the definition of $\Vset^{\max}$. $A$ is an M-matrix (since $\Lap_{\rm red}$ is a principal submatrix of a Laplacian of a connected graph and hence an M-matrix, $\Btot_{\rm red}$ is diagonal matrix with positive entries). By Lemma \ref{thm:Z}, we find that $\Vset \leq \Vset^{\max}$. Further, plugging in $\Vset=\Vset^{\max}$ in \eqref{eq:tmpLem2}, we obtain $g\br{\Vset^{\max}}\leq \Vset^{\max}$.}
\end{proof}

\begin{lemma}\label{lem:EF}
Define $\tilde{g}(\gamma)=\mathbf{B}_{\rm red}(\exp\br{\gamma}-g\br{\exp\br{\gamma}})$. Then
$2\frac{\partial \tilde{g}}{\partial \gamma}$ is equal to the matrix in the LHS of \eqref{eq:EnergyFunCondvcsa} after substituting $\Vset=\expb{\gamma}$ and $\cx_{ik}=\sqrt{\expb{\gamma_i+\gamma_k}-s_{ik}^2}$.
\end{lemma}
\begin{proof}
Using the definitions of $g,\tilde{g}$, we have
\[\tilde{g}_i\br{\Vlog}=\Btot_i \expb{\Vlog_i}-\sum_{k} B_{ik} \sqrt{\expb{\Vlog_i+\Vlog_k}-\sx_{ik}^2}-q_i\,.\]
For any $k \neq i$, we have
\[\frac{\partial \tilde{g}_i}{\partial \Vlog_k}=-\frac{B_{ik}}{2} \frac{\expb{\Vlog_i+\Vlog_k}}{\sqrt{\expb{\Vlog_i+\Vlog_k}-\sx_{ik}^2}}=-\frac{B_{ik}}{2} \frac{\sqrt{\Vset_i\Vset_k}}{\cx_{ik}}\,.\]
For $k=i$, we have
\begin{align*}
\frac{\partial \tilde{g}_i}{\partial \Vlog_i} & =\Btot_i\expb{\Vlog_i} -\sum_k \frac{B_{ik}}{2} \frac{\expb{\Vlog_i+\Vlog_k}}{\sqrt{\expb{\Vlog_i+\Vlog_k}-\sx_{ik}^2}}\\
& =\Btot_i \Vset_i -\sum_k \frac{B_{ik}}{2} \frac{\sqrt{\Vset_i\Vset_k}}{\cx_{ik}}\,.
\end{align*}
Comparing these values with the LHS of \eqref{eq:EnergyFunCondvcsa}, it is easy to see that the matrix on the LHS of \eqref{eq:EnergyFunCondvcsa} is simply $2\frac{\partial \tilde{g}}{\partial \Vlog}$.
\end{proof}}

\begin{theorem}[Knaster-Tarski Theorem \cite{tarski1955lattice}/Kantorovich Lemma \cite{Ortega}]\label{thm:Tarski}
Let $F: [a,b] \mapsto [a,b]$ be a continuous map where $a,b\in \R^n$, \djb{$a\leq b$} such that $F$ is monotone:
\[F\br{x}\geq F\br{x^\prime} \quad \forall x,x^\prime \in [a,b]\,,\,\, x\geq x^\prime.\]
Define $A=\{x\in [a,b] :F\br{x}\geq x\},B=\{x \in [a,b] :F\br{x}\leq x\}$.
\djb{$F$ has a maximal fixed point $x^\star$ ($x \leq x^\star \, \forall x \in A$) and a minimal fixed point $x_\star$ ($x_\star \leq x \, \forall x \in B$). The iteration $x^{(i+1)} \gets F\br{x^{(i)}}$ initialized at $x^{(0)} = b$ converges to $x^\star$.}
\end{theorem}

\begin{lemma}[\cite{PLEMMONS1977175}]\label{thm:Z}
If $A \in \mathbb{R}^{n \times n}$ is an irreducible Z-matrix, the following statements are equivalent: (i) $A$ is an M-matrix; (ii) there exists $x \geq 0$ such that $Ax > 0$; (iii) $-A$ is Hurwitz; (iv) $A^{-1} > 0$ entry-wise; \djb{(v) $Ax \geq 0$ entry-wise $\Rightarrow$ $x \geq 0$ entry-wise for all $x \in \mathbb{R}^n$.}
\end{lemma}

\ifCLASSOPTIONcaptionsoff
  \newpage
\fi

\end{document}